\documentclass[12pt,oneside,a4paper,reqno]{amsart}
\usepackage{latexsym,amsxtra,amsthm, amsmath, amsfonts, amscd, amssymb, verbatim}
\usepackage{amsfonts}
\usepackage[portrait, top=2.5 cm, bottom=3cm, left=3.5cm, right=2cm] {geometry} 

\usepackage{url}

\usepackage{varioref}
\usepackage{fancyhdr}
\usepackage{longtable}
\usepackage{fancybox}
\usepackage{color, graphicx}
\usepackage{verbatim}
\usepackage{setspace}
\usepackage{enumitem}

\usepackage[latin1]{inputenc}
\usepackage[english]{babel}

\newtheorem{prop} {Proposition} [section]
\newtheorem{thm}[prop] {Theorem}
\newtheorem{defi}[prop] {Definition}
\newtheorem*{defi-non}{Definition}
\newtheorem{lem}[prop] {Lemma}
\newtheorem{defi and lem}[prop] {Definition/Lemma}

\newtheorem{prop-def}[prop]{Proposition-Definition}

\newcommand{\Spec}{{\operatorname{Spec}}}

\begin{document}

\title[Invariants of models of genus one curves]{Invariants of models of genus one curves and modular forms}

\author{Manh Hung Tran}
\address{Chalmers University of Technology and University of Gothenburg \\
Sweden}\email{tranmanhhungkhtn@gmail.com}

\maketitle

\begin{abstract}
An invariant of a model of genus one curve is a polynomial in the coefficients of the model that is stable under certain linear transformations. The classical example of an invariant is the discriminant, which characterizes the singularity of models. The ring of invariants of genus one models over a field is generated by two elements. Fisher normalized these invariants for models of degree $n=2,3,4$ in such a way that these invariants are moreover defined over the integers. We provide an alternative way to express these normalized invariants using modular forms. This method relies on a direct computation for the discriminants based on their own geometric properties.

\end{abstract}

\tableofcontents

\section{Introduction}\label{section introduction}

Consider a curve $C$ given by the Weierstrass equation
\begin{equation}\label{long Weierstrass}
y^2+a_1xy+a_3y=x^3+a_2x^2+a_4x+a_6.
\end{equation}
There are two classical invariants labeled by $c_4$ and $c_6$ defined for instance in \cite[p. 42]{Silverman} as
\begin{equation}\label{c4 c6}
 c_4=b_2^2 - 24b_4,\\
 c_6=-b_2^3+36b_2b_4-216b_6,
\end{equation}
where $b_2=a_1^2+4a_2, b_4=2a_4+a_1a_3$ and $b_6=a_3^2+4a_6$. They define the discriminant $\Delta=(c_4^3-c_6^2)/1728$ with the property that $\Delta\neq 0$ if and only if the curve is non-singular. In this case, the curve $C$ is of genus one. An invariant of the model associated to \eqref{long Weierstrass} is a polynomial in their coefficients, which remains unchanged under the linear algebraic transformations mentioned in Section \ref{section genus one models and invariants}.

We want to explore this study to other models of genus one curves by a natural connection to modular forms as described in Section \ref{section invariants and modular forms}.

Let $C$ be a smooth curve of genus one over a field $K$ and suppose that $D$ is a $K$-rational divisor on $C$ of degree $n$. In case $n=1$, $C$ has a $K$-rational point so that it can be given by a Weierstrass equation \eqref{long Weierstrass}.

If $n\geq 2$, there exists a morphism $C\rightarrow \mathbb{P}^{n-1}$ defined by the complete linear system associated to $D$. This morphism is an embedding if $n\geq 3$. We define models of genus one of degrees $n\leq 5$ as follows
\begin{defi-non}\label{definition genus one models}
A genus one model of degree $n\leq 5$ is a
\begin{enumerate}[label=\alph*)]
    \item Weierstrass form if $n=1$.
    \item pair of a binary quadratic and a binary quartic if $n=2$.
    \item ternary cubic if $n=3$.
    \item pair of quadrics in four variables if $n=4$.
    \item $5\times 5$ alternating matrix of linear forms in five variables if $n=5$.
\end{enumerate}
\end{defi-non}
The equation defined by a genus one model $(p,q)$ of degree $n=2$ is $y^2+p(x,z)y=q(x,z)$. In case $n=5$, the equations defining the model are the $4\times 4$ Pfaffians of the matrix. In general, such models define smooth curves of genus one.

The classical invariants where $n=2,3,4$ were studied in \cite{Weil} and \cite[Section 3]{An}. These invariants were normalized by Fisher \cite[Sections 6,7]{Fisher} so that they are usual formulae when restricted to the Weierstrass family.

The aim of this paper is to give an alternative way to express the normalized invariants $c_4$, $c_6$ and thus $\Delta=(c_4^3-c_6^2)/1728$ for genus one models of degrees $n=2,3,4$. To do this, we establish formulae in all characteristics relating the invariants of smooth genus one models of degrees $n\leq 4$ and the corresponding Jacobians in the classical setting as in \cite{An}. More precisely, the authors in \cite[Section 3]{An} define a map $f_n$ from a smooth curve $C_{\phi}$, which is defined by a model of genus one $\phi$ of degree $n$ $(n=2,3,4)$, to the corresponding Jacobian $E_{\phi}$. In addition, they describe explicitly when char$(K)\neq 2,3$ the map and the Jacobian $E_{\phi}$ given by a Weierstrass equation of the form
$$y^2 = 4x^3 - g_2x - g_3.$$
The map $f_n$ will be described explicitly in Section \ref{section genus one models and invariants}. We construct from it in Section \ref{section discriminant and jacobians} the map $\varphi_n: X_n \rightarrow W$ from the affine space $X_n$ of genus one models of degree $n$ ($n=2,3,4$) to the space $W$ of Weierstrass forms. We first compute the normalized discriminants.

\begin{thm}\label{theorem disc}
Let $C_{\phi}$ be a curve defined by a genus one model $\phi$ of degree $n$ $(n=2,3,4)$ over a field $K$ and $\Delta_{\phi}, \Delta_{\varphi_n(\phi)}$ be the discriminants of $\phi$ and its corresponding Weierstrass form $\varphi_n(\phi)$. We have
$$\Delta_{\phi}=\alpha_n^{12} \Delta_{\varphi_n(\phi)},$$
where $\alpha_2=1, \alpha_3=1/2, \alpha_4=2$.
\end{thm}

This theorem is established directly using the singularities of genus one models. To obtain the analogous result for any invariant, we will use geometric modular forms defined in \cite{Katz}. To be precise, we will see in Section \ref{section invariants and modular forms} that it is possible to associate to a geometric modular form $\mathcal{F}$ an invariant $I_{\mathcal{F}}$ of the same weight. We will prove the following result

\begin{thm}\label{theorem invariants}
Let $C_{\phi}$ be a smooth curve of genus one defined by a model $\phi$ of degree $n$ $(n=2,3,4)$ over a field $K$ with the corresponding Jacobian $E_{\phi}$ defined by $\varphi_n(\phi)$. Let $k$ be an integer and $I_{\mathcal{F}}$ be the invariant of weight $k$ associated to a geometric modular form $\mathcal{F}$ of weight $k$, we have
$$I_{\mathcal{F}}(\phi)=\alpha_n^k I_{\mathcal{F}}(\varphi_n(\phi)),$$
where $\alpha_2=1, \alpha_3=1/2, \alpha_4=2$.
\end{thm}

The appearance of $\alpha_n$ naturally comes from the transformation property of modular forms as described in Section \ref{section invariants and modular forms}. However, we can even prove that $\alpha_n$ is a universal constant.

Recently, Fisher \cite[p. 2126]{Fisher arbitrary} have obtained a formula for the invariants $c_4$, $c_6$ and the Jacobian of smooth genus one models of arbitrary degree $n$ in characteristic 0. More details about models of genus one curves and their invariants will be discussed in Section \ref{section genus one models and invariants}.

\section{Models of genus one curves and their invariants}\label{section genus one models and invariants}

Let $X_n$ be the set of all genus one models of degree $n$ over a field $K$ (see Definition \ref{definition genus one models}), we will see later in this section that $X_n$ is an affine space of dimension $5,8,10,20,50$ for $n=1,2,3,4,5$ respectively. In \cite[Section 3]{Fisher}, the author defined natural linear algebraic groups $\mathcal{G}_n$ acting on $X_n$ ($n\leq 5$), which preserve the solutions of the models. Let $G_n$ be the commutator subgroups of $\mathcal{G}_n$ and $K[X_n]$ be the coordinate ring of $X_n$.
\begin{defi}\label{definition ring of invariants}
The ring of invariants of $X_n$ $(n\leq 5)$ over $K$ is
$$K[X_n]^{G_n}:=\{I\in K[X_n]: I\circ g=I \text{ for all } g\in G_n(\overline{K})\}.$$
The vector space of invariants of weight $k$ of $X_n$ over $K$ is defined as
$$K[X_n]_k^{G_n}:=\{I\in K[X_n]: I\circ g=(\det g)^k I \text{ for all } g\in \mathcal{G}_n(\overline{K})\}.$$
\end{defi}
The character $\det$ on $\mathcal{G}_n$ is chosen so that we get appropriate weights for the invariants and moreover (see \cite[Lemma 4.3]{Fisher})
$$K[X_n]^{G_n}=\bigoplus_{k\geq 0}K[X_n]_k^{G_n}.$$
We now describe in detail affine spaces $X_n$ for $n\leq 5$, their classical invariants and the map $f_n: C_{\phi} \rightarrow E_{\phi}$ ($n=2,3,4$) from a smooth curve defined by a model $\phi\in X_n^0:=\{\phi \in X_n \mid C_{\phi} \text{ is smooth}\}$ to its corresponding Jacobian $E_{\phi}$. This map is defined by a divisor $D$ of degree $n$ on the curve $C_{\phi}$, which is the intersection of $C_{\phi}$ with the hyperplane at infinity (see \cite[p. 305]{An}), as $f_n(P)=nP - D$. It is given explicitly when char$(K)\neq 2,3$ as below. In addition, we introduce in each degree the natural non-zero regular 1-form $\omega_{\phi}$ defined in \cite[Section 5.4]{Fisher} for a smooth curve of genus one $C_{\phi}$. This 1-form is useful for relating invariants and modular forms as we will see in Section \ref{section invariants and modular forms}.

\subsection{Models of degree $n=1$}\label{weierstrass pair}
The space $X_1$ is of dimension 5 where each model corresponding to \eqref{long Weierstrass} can be identify with the point $(a_1,a_2,a_3,a_4,a_6)$ in $\mathbb{A}^5$. Their invariants $c_4,c_6$ are defined in the usual way as in \eqref{c4 c6}. The natural regular 1-form on a smooth curve $C_{\phi}$ defined by model $\phi$ is:
$$\omega_{\phi}:=\frac{dx}{2y+a_1x+a_3}.$$
\subsection{Models of degree $n=2$}
The equation of a genus one model of degree $n=2$ is written as:
\begin{equation}\label{n=2 generalized}
y^2+(\alpha_0 x^2+\alpha_1 xz+\alpha_2 z^2)y=ax^4+bx^3 z+cx^2 z^2+dx z^3+ez^4.
\end{equation}
Each model in $X_2$ corresponds to a point $(\alpha_0,\alpha_1,\alpha_2,a,b,c,d,e)\in \mathbb{A}^8$ and thus \linebreak dim$(X_2)=8$. Moreover, if char$(K)\neq 2$ or $3$, we can rewrite \eqref{n=2 generalized} as:
\begin{equation}\label{n=2}
y^2=ax^4+bx^3z+cx^2z^2+dxz^3+ez^4.
\end{equation}
As in \cite[Section 3.1]{An}, the model $\phi$ defined by \eqref{n=2} has two classical invariants:
$$i=(12ae-3bd+c^2)/12,$$
\begin{equation}\label{i j}
j=(72ace-27ad^2-27b^2e+9bcd-2c^3)/432
\end{equation}
and a corresponding Weierstrass equation is:
\begin{equation}\label{jacobian n=2}
y^2=4x^3-ix-j.
\end{equation}
The map $f_2$ from a smooth curve $C_{\phi}$ defined by \eqref{n=2} to the Jacobian $E_{\phi}$ defined by \eqref{jacobian n=2} is given as in \cite[(3.3)]{An} by:
$$f_2(x,y,z)=\left(\frac{g(x,z)}{(yz)^2},\frac{h(x,z)}{(yz)^3}\right),$$
where
$$g(x,z)=\frac{1}{144}(q_{xz}^2 - q_{xx} q_{zz}), \text{ } h(x,z)=\frac{1}{8}\left|
                                                                      \begin{array}{cc}
                                                                        q_x & q_z \\
                                                                        g_x & g_z \\
                                                                      \end{array}
                                                                    \right|
$$
with $q$ being the binary quartic on the right hand side of \eqref{n=2}.

The corresponding results to the generalized equation \eqref{n=2 generalized} can be found in \cite[p. 766]{Cremona} by completing the square. We define the regular 1-form for a smooth curve $C_{\phi}$ defined by a model $\phi$ corresponding to \eqref{n=2 generalized} as:
$$\omega_{\phi}:=\frac{z^2d(x/z)}{2y+p(x,z)}.$$
Here $p(x,z)=\alpha_0 x^2 + \alpha_1 xz + \alpha_2 z^2$.
\subsection{Models of degree $n=3$}
In case $n=3$, a genus one model $\phi$ is a ternary cubic:
\begin{equation}\label{n=3}
ax^3+by^3+cz^3+a_2x^2y+a_3x^2x+b_1xy^2+b_3y^2z+c_1xz^2+c_2yz^2+mxyz
\end{equation}
with two classical invariants $S,T$ defined in \cite[p. 309-310]{An} and we can see that dim$(X_3)=10$. A corresponding Weierstrass equation is also given there by:
\begin{equation}\label{jacobian n=3}
y^2=4x^3+108Sx-27T.
\end{equation}
The map $f_3$ from a smooth curve $C_{\phi}$ defined by the model \eqref{n=3} to the Jacobian $E_{\phi}$ defined by \eqref{jacobian n=3} is given as in \cite[(3.9)]{An} by:
$$f_3(x,y,z)=\left(\frac{\Theta(x,y,z)}{H(x,y,z)^2},\frac{J(x,y,z)}{H(x,y,z)^3}\right),$$
where
$$H=\frac{1}{216}\left|
                   \begin{array}{ccc}
                     \phi_{xx} & \phi_{xy} & \phi_{xz} \\
                     \phi_{yx} & \phi_{yy} & \phi_{yz} \\
                     \phi_{zx} & \phi_{zy} & \phi_{zz} \\
                   \end{array}
                 \right|, \text{ } J=-\frac{1}{9}\left|\frac{\partial (\phi,H,\Theta)}{\partial(x,y,z)}\right|
$$
and $\Theta$ is the covariant defined in \cite[p. 308]{An}. The regular 1-form on the smooth curve $C_{\phi}$ is defined as
$$\omega_{\phi}:=\frac{x^2d(y/x)}{\partial \phi/ \partial z}.$$
\subsection{Models of degree $n=4$}
We relate to the case $n=2$ as follows: if the model $\phi$ is given by a pair of quadrics $q_1,q_2\in K[x_0,x_1,x_2,x_3]$ (hence dim$(X_4)=20$), we can write $q_1=\overline{x} A \overline{x}^T$ and $q_2=\overline{x} B \overline{x}^T$ for two symmetric $4\times 4$ matrices $A,B$ with $\overline{x}=(x_0,x_1,x_2,x_3)$. The invariants of $\phi$ are then defined by the invariants of the quartic:
\begin{equation}\label{n=4}
\det(xA+zB)=ax^4+bx^3z+cx^2z^2+dxz^3+ez^4
\end{equation}
as in the case $n=2$. A corresponding Weierstrass equation is thus given in the form \eqref{jacobian n=2} with $i,j$ defined by the coefficients of the model \eqref{n=4} as in \eqref{i j}. The explicit map $f_4$ from a smooth curve defined by \eqref{n=4} to the Jacobian is given in \cite[(3.12)]{An} as:
$$f_4(x_0,x_1,x_2,x_3)=\left(\frac{g}{J^2}, \frac{h}{J^3}\right),$$
where $g,h,J$ are defined as in \cite[Section 3.3]{An}. The regular 1-form for the smooth curve $C_{\phi}$ is defined as:
$$\omega_{\phi}:=\frac{x_0^2 d(x_1/x_0)}{(\partial q_1 / \partial x_3)(\partial q_2 / \partial x_2) - (\partial q_1 / \partial x_2)(\partial q_2 / \partial x_3)}.$$
\subsection{Models of degree $n=5$}
A genus one model of degree $n=5$ is a $5\times 5$ alternating matrix of linear forms in five variables. The equations defined by this model are the $4\times 4$ Pfaffians of the matrix. Here a square matrix is called alternating if it is skew-symmetric and all of its diagonal entries are zero. The reader can have a look at \cite[Section 5.2]{Fisher} and \cite{Fisher Pfaffians} for reference. We can check that dim$(X_5)=50$. Fisher \cite[p. 770]{Fisher} also defines a regular 1-form to the case $n=5$ when char$(K)\neq 2$. This restriction on char$(K)$ is removed in \cite[Remark 7.6]{Sadek Fisher}.

Observe that if $I$ is an invariant of weight $k$ then so is $\lambda I$ for any constant $\lambda\in K^*$. We want to normalize the invariants so that they have appropriate formulae in any characteristic.

\subsection{Normalized invariants}

The author in \cite[Theorem 10.2]{Fisher} proves that $K[X_n]^{G_n}$ is isomorphic to $K[X_1]^{G_1}$ $(n\leq 5)$ in any characteristic. This extends the invariants $c_4,c_6,\Delta\in \mathbb{Z}[X_1]$  of $K[X_1]^{G_1}$ defined in \eqref{c4 c6} to the corresponding ones in $K[X_n]^{G_n}$ ($n\leq 5$) denoted again by $c_4,c_6,\Delta.$ In fact, we have (see \cite[Lemma 4.15 and remark 4.16]{Fisher})
\begin{lem}\label{primitivity of c4, c6 and D}
The invariants $c_4,c_6$ and $\Delta$ are primitive polynomial in $\mathbb{Z}[X_n]$ for any $n\leq 5$.
\end{lem}
Thus it is possible to normalize these invariants up to sign. Furthermore, the two invariants $c_4,c_6\in K[X_n]^{G_n}$ $(n\leq 5)$ constructed above define the suitable discriminant $\Delta=(c_4^3-c_6^2)/1728$ in the following way:

\begin{defi and lem}\label{definition disc}
Let $R$ be a Noetherian (unital commutative) ring and $C_{\phi}$ be a curve over $R$ defined by some genus one model $\phi$ of degree $n\leq 5$. There exists the discriminant $\Delta$, which is a universal polynomial with integer coefficients, is defined such that $\Delta_{\phi} \in R^*$ if and only if $C_{\phi}$ is non-singular over $R$. Here $R^*$ is the group of units of $R$.
\end{defi and lem}
\begin{proof}
Fisher \cite[Theorem 4.4]{Fisher} shows the properties of the discriminant of genus one models of degree $n\leq 5$ over a field $K$, but it is indeed equivalent to the Definition/Lemma \ref{definition disc}.
\end{proof}
He also proves there that if char$(K)\neq 2$ or $3$, then the invariants $c_4(\phi),c_6(\phi)$ provide for the smooth genus one curve $C_{\phi}$ defined by a model $\phi$ of degree $n\leq 5$ the Jacobian
$$y^2=x^3-27c_4(\phi)x-54c_6(\phi).$$
The main ingredient of the proof of the above result is that, a pair $(C_{\phi},\omega_{\phi})$ of a model $\phi\in X_n^0$ is isomorphic to a pair of the form given in Section \ref{weierstrass pair}. Then the invariants $c_4(\phi),c_6(\phi)$ of $\phi$ are determined by the ones of that pair as in \eqref{c4 c6} (see \cite[Definition 2.1 and Proposition 5.23]{Fisher}).

Then, in \cite[Section 7]{Fisher}, he gives the normalized formulae to the invariants $c_4,c_6$ of genus one models of degrees $n=2,3,4$. More precisely, in the case $n=2$, the normalized invariants of the model corresponding to \eqref{n=2} are
$$c_4=2^4(12ae-3bd+c^2),$$
\begin{equation}\label{n=2 geometric}
c_6=2^5(72ace-27ad^2-27b^2e+9bcd-2c^3).
\end{equation}
In comparison with the classical case \eqref{i j}, we have $c_4=2^6 3 i$ and $c_6=2^9 3^3 j$. The normalized invariants of the generalized model corresponding to \eqref{n=2 generalized} can be found in \cite[p. 766]{Cremona} by reducing to the form \eqref{n=2} from a completing square.

When $n=3$, the normalized invariants of the model \eqref{n=3} are
$$c_4=-216abcm+144abc_1c_2+144acb_1b_3+...-8a_3b_3m^2+16b_1^2c_1^2-8b_1c_1m^2+m^4,$$
\begin{equation}\label{n=3 geometric}
c_6=5832a^2b^2c^2-3888a^2bcb_3c_2+864a^2bc_2^3+...+64b_1^3c_1^3-48b_1^2c_1^2m^2+12b_1c_1m^4-m^6.
\end{equation}
We have $c_4=-2^4 3^4 S, c_6=2^3 3^6 T$ in comparing with the classical invariants $S,T$ defined in \cite[p. 309-310]{An}.

When $n=4$, we relate to the quartic as in \eqref{n=4} and then the invariants are
$$c_4=12ae-3bd+c^2,$$
\begin{equation}\label{n=4 geometric}
c_6=\frac{1}{2}(72ace-27ad^2-27b^2e+9bcd-2c^3).
\end{equation}
Thus in comparing with the classical case, we have $c_4=2^{10} 3 i$ and $c_6=2^{15} 3^3 j$.

In the case of plane cubics, these invariants were normalized before in \cite[p. 367-368]{Villegas} where the authors explicitly provide the corresponding Weierstrass equations in any characteristic. Strictly speaking, they associate to a ternary cubic $\phi$ a Weierstrass form $\phi^*$ associated to \eqref{long Weierstrass} with the coefficients determined by the coefficients of $\phi$. Then they observe that $(\phi^*)^*=\phi^*$ and thus naturally define:
$$c_4(\phi):=c_4(\phi^*),c_6(\phi):=c_6(\phi^*),\Delta_{\phi}:=\Delta_{\phi^*}.$$
In Section \ref{section invariants and modular forms}, we will provide a new way for expressing these normalized invariants using modular forms. Before that, however, we will establish the normalized discriminants of genus one models directly by using the singularities of the models in the next section.

\section{Discriminants of genus one curves and their Jacobians}\label{section discriminant and jacobians}

The goal of this section is to prove Theorem \ref{theorem disc}. Fix $n\leq 5$, let $X_n$ be the affine space of all genus one models $\phi$ of degree $n$, let $W$ be the affine space of Weierstrass forms $y^2-4x^3 + g_2 x +g_3$ if char$(K)\neq 2,3$ and $W=X_1$ otherwise. We define the map $\varphi_n:X_n\rightarrow W$ based on the discussion about the map $f_n$ in Section \ref{section genus one models and invariants} for $n=2,3,4$. More precisely, for any smooth model $\phi\in X_n^0$, we have a map $f_n: C_{\phi} \rightarrow E_{\phi}$ from the smooth curve $C_{\phi}$ to its Jacobian $E_{\phi}$ coming from a divisor of degree $n$. We define the image $\varphi_n(\phi)$ of $\phi$ to be the model in $W$ defining $E_{\phi}$. This gives us a map $X_n^0 \rightarrow W$ which extends uniquely to a map $\varphi_n: X_n \rightarrow W$.

The map $\varphi_n$ is given explicitly in case char$(K)\neq 2$ or $3$, which also applies to singular models as follows. The map $\varphi_2$ sends a model $\phi$ of degree 2 corresponding to \eqref{n=2 generalized} to the Weierstrass form defined by \eqref{jacobian n=2} of the model corresponding to \eqref{n=2} obtained from $\phi$ after a completing square. The map $\varphi_3$ sends a model $\phi$ of degree 3 of the form \eqref{n=3} to the Weierstrass form defined by \eqref{jacobian n=3}. The map $\varphi_4$ is defined by relating to the case $n=2$. More precisely, $\varphi_4$ sends a model $\phi$ of degree 4 given by a pair of quadrics $(q_1,q_2)$ to the Weierstrass form defined by \eqref{jacobian n=2} obtained from the quartic \eqref{n=4} as in the case $n=2$. We denote by $E_{\phi}$ the curve given by the corresponding Weierstrass form $\varphi_n(\phi)\in W$ of a model $\phi\in X_n$.

This map sends non-singular curves to non-singular curves and singular curves to singular curves when char$(K)\neq 2,3$. We know that there exist discriminants (see Definition/Lemma \ref{definition disc}) $\Delta_{X_n}$ in $X_n$ and $\Delta_W$ in $W$ parametrizing singular curves and they are both geometrically irreducible polynomials (see \cite[Proposition 4.5]{Fisher}). The discriminant $\Delta_{\phi}$ of $\phi$ is determined by evaluating $\Delta_{X_n}$ at coefficients of the model $\phi$. i.e., $\Delta_{\phi}=\Delta_{X_n}(\phi)$.

We denote by $V(p)$ the set of points in $\mathbb{P}_K^m$ vanished at $p$ for any homogeneous polynomial $p\in K[x_0,...,x_m]$. Observe that $\Delta_{X_n}$ and the pull back $\varphi_n^*(\Delta_W)$ of $\Delta_W$ have the same vanishing property:
\begin{equation}\label{vanishing}
V(\Delta_{X_n})=V(\varphi_n^*(\Delta_W)),
\end{equation}
where $\varphi_n^*$ is defined such that $\varphi_n^*(\Delta_W)(\phi)=\Delta_W(\varphi_n(\phi))$ for any model $\phi$. We have the following:
\begin{prop}\label{proposition constant c}
For $n\leq 4$, there exists a constant $c\in K^*$ such that
$$\Delta_{X_n}=c \varphi_n^*(\Delta_W).$$
\end{prop}
\begin{proof}
For an ideal $J\subset K[x_0,...,x_m]$, we denote by $\sqrt{J}$ the radical ideal of $J$ over $\overline{K}$. From \eqref{vanishing} and Hilbert's Nullstellensatz, we obtain that $\sqrt{(\Delta_{X_n})}=\sqrt{(\varphi_n^*(\Delta_W))}$. Note that $\Delta_{X_n}$ is geometrically irreducible as mentioned above. Then $\sqrt{(\Delta_{X_n})}=(\Delta_{X_n})$ (which is an ideal over $\overline{K}$) and thus there exist constants $k\in \mathbb{Z}$ and $c\in \overline{K}^*$ such that $\Delta_{X_n}^k =c \varphi_n^*(\Delta_W)$. Since both $\Delta_{X_n}^k$ and $\varphi_n^*(\Delta_W)$ are defined over $K$, we conclude that $c\in K^*$.

We will prove that $k=1$ by comparing the degrees of $\Delta_{X_n}$ and $\varphi_n^*(\Delta_W)$. We first consider the case $n=3$. By \cite[Example 1.8]{Benoist}, we know that $\Delta_{X_n}$ is a homogeneous polynomial of degree 12 with respect to the coefficients of the plane cubics. Besides, $\varphi_3^*(\Delta_W)$ is also of degree 12 in the coefficients of the cubics so that it has the same degree with $\Delta_{X_n}$.

If $n=4$, by \cite[Example 1.10]{Benoist} we see that $\Delta_{X_n}$ is a homogeneous polynomial of degree 24 in the coefficients of the two quadratic forms defining the models. The degree of $\varphi_4^*(\Delta_W)$ is also 24.

The case $n=2$ is different since the model $y^2+p(x,z)y=q(x,z)$ is no longer homogeneous. Here $p,q$ are homogeneous polynomials of degrees $2,4$ respectively. If char$(K)\neq 2$, this model can be brought to $y^2=h(x,z)$ by completing square, where $h=\frac{p^2}{4}+q$. The singular locus of this latter model is $\{2y=h_x=h_z=0\}$, which is equal to $\{h_x=h_z=0\}$ if char$(K)\neq 2$. Consequently, the discriminant of the model $(p,q)$ above is the discriminant of the quartic $h$ from Definition/Lemma \ref{definition disc} (up to some power of 2). Hence, by \cite[Example 1.8]{Benoist} again, we know that $\Delta_{X_n}(p,q)$ is of degree 6 with respect to the coefficients of $h$. We can check that $\varphi_2^*(\Delta_W)(p,q)$ is of degree 6 in terms of coefficients of $h$ as well. Thus $k=1$ for $n=2,3,4$.
\end{proof}

We actually have more information about the constant $c$ by looking at models with integer coefficients.

\begin{prop}\label{proposition power of 2}
For $n=2,3,4$, the constant $c$ in Proposition \ref{proposition constant c} can be expressed as $\pm 2^a$ for some $a\in \mathbb{Z}$. Moreover,
$$\left\{
  \begin{array}{ll}
    a=0,  & \hbox{If $n=2$;} \\
    a=-12,  & \hbox{If $n=3$;} \\
    a=12,  & \hbox{If $n=4$.}
  \end{array}
\right.
$$
\end{prop}

We will see later in Theorem \ref{theorem invariants of C and E} that one can exclude the minus sign of the constant $c$ above and thus obtain Theorem \ref{theorem disc}.

\begin{proof}
We consider the curve $C_{\phi}$ defined by a model $\phi$ with integer coefficients so that we can make use of reduction modulo prime numbers. Note that $\Delta_{\phi} \in \mathbb{Z}$ by Lemma \ref{primitivity of c4, c6 and D}. Since the map $\varphi_n$ sends non-singular curves to non-singular curves over characteristics not 2 and 3, the constant $c$ is of the form $\pm 2^a 3^b$ for some integers $a,b$. To compute the powers of 2 and 3, we need to compare $\Delta_{\phi}$ and $\Delta_{\varphi_n(\phi)}$ over $\mathbb{Z}_{2}$ and $\mathbb{Z}_{3}$ (see Definition/Lemma \ref{definition disc}). Here $\Delta_{\varphi_n(\phi)}=\Delta_W(\varphi_n(\phi))$ is the discriminant of the Weierstrass form $\varphi_n(\phi)$ of $\phi$. Since $\Delta_{\phi}=c\Delta_{\varphi_n(\phi)}$, we get the following $p$-adic valuation identity for a prime number $p$:
\begin{equation}\label{valuation}
v_p(\Delta_{\phi})=v_p(c)+v_p(\Delta_{\varphi_n(\phi)}).
\end{equation}
Suppose that $C_{\phi}$ is non-singular over $\mathbb{F}_2$ and $\mathbb{F}_3$, then $v_2(\Delta_{\phi})=v_3(\Delta_{\phi})=0$ and we get from \eqref{valuation} that $v_2(c)=-v_2(\Delta_{\varphi_n(\phi)}), v_3(c)=-v_3(\Delta_{\varphi_n(\phi)})$. So to find $c$, we just need to compute $\Delta_{\varphi_n(\phi)}$ in some special case in which $C_{\phi}$ is non-singular over $\mathbb{F}_2$ and $\mathbb{F}_3$.

In the case $n=2$, we consider the genus one model $\phi$ with equation
$$y^2+yz^2=x^4+x^3z+x^2z^2.$$
The corresponding Weierstrass equation of $\varphi_2(\phi)$ is:
$$y^2=4x^3-\frac{1}{3}x-\frac{37}{1728}$$
and $\Delta_{\varphi_2(\phi)}=101.$ The model $\phi$ is non-singular over both $\mathbb{F}_2$ and $\mathbb{F}_3$. Hence $v_2(c)=0, v_3(c)=0$ and thus $a=0, b=0$.

If $n=3$, we consider the curve $C_{\phi}$ given by $y^2z+yz^2-x^3=0$ with the corresponding equation of $\varphi_3(\phi): y^2=4x^3+1$ and we have $\Delta_{\varphi_3(\phi)}=-2^{12} 3^3$. The curve $C_{\phi}$ is non-singular over $\mathbb{F}_2$ and thus $v_2(c)=-12$ or $a=-12$. To compute the power of 3, we consider the following curve $y^2z-x^3-xz^2=0$ which is non-singular over $\mathbb{F}_3$ with the corresponding equation of $\varphi_3(\phi): y^2=4x^3+4x$ and $\Delta_{\varphi_3(\phi)}=-2^{18}$. Therefore, we have $v_3(c)=0$ or $b=0$.

When $n=4$, we consider the curve $C_{\phi}$ given by the following complete intersection of two quadratic forms
$$\left\{
    \begin{array}{ll}
      x_0x_1+x_0x_2+x_2x_3=0 \\
      x_0x_3+x_1x_2+x_1x_3=0
    \end{array}
  \right.
$$
and the corresponding Weierstrass equation
$$y^2=4x^3-\frac{1}{2^{10} 3}x + \frac{161}{2^{15} 3^3}$$
computed by $\varphi_4$ with $\Delta_{\varphi_4(\phi)}=-3.5/2^{12}$. It is possible to check that $C_{\phi}$ is non-singular over $\mathbb{F}_2$. This implies that $v_2(c)=12$ and hence $a=12$. Observe that this curve is singular over $\mathbb{F}_3$ since $(1,1,1,1)$ is a singular point modulo 3. So to compute the power of 3, we need to look at another example. For instance, we consider the following complete intersection which is non-singular over $\mathbb{F}_3$
$$\left\{
    \begin{array}{ll}
      x_0^2+x_1^2+x_2^2+3x_3^2=0 \\
      x_0^2+2x_1^2+3x_2^2+5x_3^2=0
    \end{array}
  \right.
$$
with the corresponding equation of $\varphi_4(\phi): y^2=4x^3-x$ and we obtain in this case that $\Delta_{\varphi_4(\phi)}=2^{12}$. This means that $v_3(c)=0$ and thus $b=0$. This completes the proof of Proposition \ref{proposition power of 2}.
\end{proof}

\section{Invariants and modular forms}\label{section invariants and modular forms}

We now in this section study all invariants of genus one models. To do this, we first give a brief introduction to the theory of modular forms and the connection to invariants. The goal is to establish Theorem \ref{theorem invariants} by proving Theorem \ref{theorem invariants of C and E}, which is the main result of this section.

\subsection{Weakly holomorphic and geometric modular forms}

A weakly holomorphic modular form $F$ of weight $k\in \mathbb{Z}$ is a holomorphic function on the upper half-plane $\mathbb{H}=\{\tau \in \mathbb{C} \mid  \text{Im}\tau >0\}$, that is meromorphic at $\infty$ and satisfies the equation
$$F\left(\frac{a\tau+b}{c\tau+d}\right)=(c\tau+d)^k F(\tau) \text{ for all } \left(
    \begin{array}{cc}
      a & b \\
      c & d \\
    \end{array}
  \right) \in SL(2,\mathbb{Z}),$$
where
$$SL(2,\mathbb{Z}):=\left\{ \left(
                              \begin{array}{cc}
                                a & b \\
                                c & d \\
                              \end{array}
                            \right)\mid a,b,c,d \in \mathbb{Z}, ad-bc=1
 \right\}.$$
Denote by $M_k^{!}(\mathbb{C})$ the space of weakly holomorphic modular forms of weight $k$ and $M^{!}(\mathbb{C})$ the graded algebra
$$M^!(\mathbb{C}):=\bigoplus_{k\in \mathbb{Z}}M_k^!(\mathbb{C}).$$
$F$ is called holomorphic if it is holomorphic at $\infty$, i.e., $F$ has a Fourier expansion
$$F(\tau)=\sum_{n=0}^{\infty} a_n e^{2\pi i n \tau}$$
which is absolutely convergent for each $\tau \in \mathbb{H}$. We denote by $M_{k}(\mathbb{C})$ the space of holomorphic modular forms of weight $k$ and $M(\mathbb{C})$ the graded algebra
$$M(\mathbb{C}):=\bigoplus_{k\geq 0}M_{k}(\mathbb{C}).$$
One of the most important examples of holomorphic modular forms is the Eisenstein series $G_{2k}$, which is of weight $2k$, is defined for an integer $k\geq 2$ as
\begin{equation}\label{G_2k}
G_{2k}=\sum_{(m,n)\in \mathbb{Z}^2 \backslash (0,0)}\frac{1}{(m+n\tau)^{2k}}.
\end{equation}
We usually use the following normalized notation of the Eisenstein series
\begin{equation}\label{E_2k}
E_{2k}:=\frac{G_{2k}}{2\zeta (2k)}=1-\frac{4k}{B_{2k}}\sum_{n=1}^{+\infty} \sigma_{2k-1}(n) q^n,
\end{equation}
where $\zeta$ is the Riemann zeta function, $B_{2k}$ are the Bernoulli numbers, $\sigma$ is the divisor sum function and $q=e^{2\pi i \tau}$.

Next step is to follow \cite[p. 9,10]{Katz} to introduce the notion of geometric modular forms. Here an elliptic curve $E$ over a scheme $S$ is a proper smooth morphism $\pi:E\rightarrow S$ whose generic fibers are connected smooth curves of genus one together with a section $e: S\rightarrow E$.
\begin{defi}\label{defi:geommodularform}
A geometric modular form of weight $k\in \mathbb{Z}$ over a scheme $S$ is a rule $\mathcal{F}$ which assigns to every pair $(E/R,\omega)$ of an elliptic curve $\pi: E \rightarrow R$ over $S$ and a basis $\omega$ of $\pi_* \Omega_{E/R}^1$ an element $\mathcal{F}(E/R,\omega)\in R$ such that\\
\begin{enumerate}[label=\alph*)]
\item $\mathcal{F}(E/R, \omega)$ depends only on the $R$-isomorphism class of the pair $(E/R,\omega)$.
\item  For any $\lambda\in R^*$ we have $\mathcal{F}(E,\lambda\omega)=\lambda^{-k}\mathcal{F}(E,\omega)$.
\item  $\mathcal{F}(E'/R',\omega_{R'})=\psi(\mathcal{F}(E/R,\omega))$ for any morphism $\psi: R \rightarrow R'$, i.e., $\mathcal{F}$ commutes with arbitrary base change. Here $(E'/R',\omega_R')$ is the base change of $(E/R,\omega)$ along $\psi$.
\end{enumerate}
\end{defi}
We adopt the same definition if we only assume that $E/R$  is a smooth genus one curve over $R$ by the following lemma, which is surely known to the experts but the author was unable to locate it in the literature.
\begin{lem}
There is a natural correspondence between:
\begin{itemize}
    \item Geometric modular forms of weight $k$ for elliptic curves over a scheme $S$.
    \item Geometric modular forms of weight $k$ for smooth genus one curves over a scheme $S$.
\end{itemize}
\end{lem}
\begin{proof}
Suppose first that we are given a geometric modular form for curves of genus one, $\mathcal{F}$. An elliptic curve $E/R$ is a pair $(C,e)$ where $C$ is a smooth genus one curve $C$ over $R$ and a section of $e: \Spec(R) \to C$. We can simply forget the section and set $\mathcal{F}(E/R, \omega) := \mathcal{F}(C/R, \omega)$. This will satisfy all the properties in the definition because $\mathcal{F}(C/R, \omega)$ does.

We have to show the converse, and suppose we are given a geometric modular form $\mathcal{F}$ for elliptic curves. Locally for the \'etale topology (see \cite[17.16.3 (ii)]{EGAIV}), there are \'etale ring extensions $\psi_i : R \to R_i$ such that $C_i := C \otimes_R R_i$ admits a section, $e$, over $R_i$, giving an elliptic curve $E_i/R_i$. We argue that $\mathcal{F}(E_i/R_i, \omega_i)$ is independent of the choice of section. One can compare $e$ with another choice of section $P_i: \Spec(R_i) \to C_i$ by the $R_i$-isomorphism given by translation by $P_i$, $\tau_{P_i}: (C_i, e_i) \simeq (C_i, P_i)$. This is an isomorphism of elliptic curves. Since the differential $\omega_i$ is invariant under $\tau_{P_i}$, the property $a)$ in Definition \ref{defi:geommodularform} shows that $\mathcal{F}((C_i, e), \omega_i) = \mathcal{F}((C_i, P_i), \omega_i) $.

We then set $\alpha_i := \mathcal{F}(E_i/R_i, \omega_i) \in R_i$,  independent of any choice of section. Let $R_j$ be another (local) ring extension such that $C_j/R_j$ admits a section, and denote by $\psi_{ij}: R_i \to R_{ij} = R_i \otimes_R R_j$ the natural map. Then $\psi_{ji}(\alpha_j)= \alpha_{ij} = \psi_{ij}(\alpha_i)$ by the property $c)$ of Definition \ref{defi:geommodularform}. By \'etale descent we obtain an element $\alpha \in R$ such that for all the maps $\psi_i :R \to R_i$, $\psi_i(\alpha) = \alpha_i$. It is a straightforward verification that $\mathcal{F}(C/R, \omega) = \alpha$ satisfies the definition of a geometric modular form of a genus one curve.
\end{proof}

Denote by $\mathcal{M}_k^!(R)$ the $R$-module of geometric modular forms of weight $k$ over a ring $R$ and $\mathcal{M}^!(R)$ the graded algebra
$$\mathcal{M}^!(R):=\bigoplus_{k\in \mathbb{Z}}\mathcal{M}_k^!(R).$$
As in \cite[p. 10]{Katz}, a geometric modular form $\mathcal{F}$ has its $q$-expansion as an element of $\mathbb{Z}((q))\otimes_{\mathbb{Z}} R$ obtained by evaluating on the pair $(\text{Tate}(q), \omega)_R$ consisting of the Tate curve and its canonical differential. $\mathcal{F}$ is called holomorphic if it is holomorphic at $\infty$, i.e., its $q$-expansion lies in $\mathbb{Z}[[q]]\otimes_{\mathbb{Z}}R$.

We denote by $\mathcal{M}_k(R)$ the $R$-module of holomorphic geometric modular forms of weight $k$ over a ring $R$ and $\mathcal{M}(R)$ the graded algebra
$$\mathcal{M}(R):=\bigoplus_{k\geq 0}\mathcal{M}_k(R).$$
We have in addition the relation
\begin{equation}\label{M and M!}
\mathcal{M}^!(R)=\mathcal{M}(R)[\mathcal{D}^{-1}],
\end{equation}
where $\mathcal{D}$ is the cusp form of weight 12 defined below.

It turns out that we can identify weakly holomorphic and geometric modular forms over $\mathbb{C}$ from the following discussion in \cite[p. 91]{Katz}. Let $C_{\tau}:= \mathbb{C}/(\mathbb{Z}+\tau \mathbb{Z})$ for any $\tau \in \mathbb{H}$. For any geometric modular form $\mathcal{F}\in \mathcal{M}_k^!(\mathbb{C})$, we can define the corresponding weakly holomorphic modular form of the same weight
\begin{equation}\label{analytic and geometric}
F(\tau)=\mathcal{F}(C_{\tau},2\pi i \text{ } dz)
\end{equation}
with $dz$ being the canonical differential on $\mathbb{C}$. Then the map $\mathcal{F}\mapsto F$ is an isomorphism $\mathcal{M}_k^!(\mathbb{C}) \cong M_k^!(\mathbb{C})$.

Observe that for $k=2,3$ the quotient $4k/B_{2k}\in \mathbb{Z}$ and thus the Eisenstein series $E_4, E_6$ in \eqref{E_2k} are defined over $\mathbb{Z}$. By the $q$-expansion principle as in \cite[Corollary 1.9.1]{Katz}, the corresponding geometric modular forms $\mathcal{E}_{2k}$ of $E_{2k}$ of weight $2k$ ($k=2,3$) are also defined over $\mathbb{Z}$.

It is known that the ring of holomorphic geometric modular forms $\mathcal{M}(\mathbb{Z})$ over $\mathbb{Z}$ is generated by $\mathcal{E}_4,\mathcal{E}_6$ and the cusp form $\mathcal{D}$ satisfying $1728\mathcal{D}=\mathcal{E}_4^3 - \mathcal{E}_6^2$. More precisely, we have (see \cite[Proposition 6.1]{Deligne})
$$\mathcal{M}(\mathbb{Z})\cong \mathbb{Z}[\mathcal{E}_4,\mathcal{E}_6,\mathcal{D}]/(\mathcal{E}_4^3-\mathcal{E}_6^2-1728 \mathcal{D}).$$
If char$(K)\neq 2$ or $3$, then 1728 is invertible over $K$ and thus $\mathcal{M}(K)=K[\mathcal{E}_4,\mathcal{E}_6]$. In this case, $\mathcal{E}_4$ and $\mathcal{E}_6$ are algebraically independent since so are $E_4$ and $E_6$.

We have a type of geometric modular forms in any characteristic $p$ called the Hasse invariants defined for instance in \cite{Katz} (p. 29). For any prime number $p$, the Hasse invariant $\mathcal{A}_p$ is a geometric modular form over $\mathbb{F}_p$ of weight $p-1$, which satisfy a certain property. The Hasse invariant $\mathcal{A}_p$ has $q$-expansion equal to 1 in $\mathbb{F}_p[[q]]$. For any prime $p>3$, we have $\mathcal{A}_p=\mathcal{E}_{p-1}$ (mod $p$) since they are both geometric modular forms of the same weight $p-1$ with the same $q$-expansions (see \cite{Katz} (p. 30)).

The structure of the graded ring of holomorphic geometric modular forms can be summarized from Propositions 6.1, 6.2, Remark 6.3 and the formula (8.4) in \cite{Deligne} as follows:
\begin{prop}\label{proposition ring of geometric modular forms}
The graded ring $\mathcal{M}(K)$ of holomorphic geometric modular forms over a field $K$ is
$$\left\{
  \begin{array}{ll}
    K[\mathcal{E}_4,\mathcal{E}_6], & \hbox{if char$(K)\neq 2,3$;} \\
    K[\mathcal{A}_2,\mathcal{D}], \mathcal{E}_4=\mathcal{A}_2^4 \text{ and } \mathcal{E}_6=-\mathcal{A}_2^6, & \hbox{if char$(K)=2$;} \\
    K[\mathcal{A}_3,\mathcal{D}], \mathcal{E}_4=\mathcal{A}_3^2 \text{ and } \mathcal{E}_6=-\mathcal{A}_3^3, & \hbox{if char$(K)=3$.}
  \end{array}
\right.
$$
\end{prop}

\subsection{Invariants and modular forms}
Similarly, the structure of the graded ring of invariants of $X_n$ can be summarized from \cite[Theorem 4.4, Lemma 10.1, Theorem 10.2]{Fisher} as below. Here $c_4,c_6$ are the usual invariants defined in Section \ref{section genus one models and invariants} and $a_1,b_2$ are the invariants of weight 1,2 respectively defined as in \cite[Theorem 10.2]{Fisher}.
\begin{prop}\label{proposition ring of invariants}
The ring of invariants $K[X_n]^{G_n}$ of $X_n$ $(n\leq 5)$ over a field $K$ is
$$\left\{
  \begin{array}{ll}
    K[c_4,c_6], & \hbox{if char$(K)\neq 2,3$;} \\
    K[a_1,\Delta], & \hbox{if char$(K)=2$;} \\
    K[b_2,\Delta], & \hbox{if char$(K)=3$.}
  \end{array}
\right.
$$
\end{prop}
The algebraic independence of the invariants $c_4$ and $c_6$ if char$(K)\neq 2$ or 3, $a_1$ and $\Delta$ if char$(K)=2$, $b_2$ and $\Delta$ if char$(K)=3$ is clear in case $n=1$. Thus they are algebraically independent for all $n\leq 5$.

We can link modular forms to invariants. To see this, we first recall a result from \cite[Proposition 5.19]{Fisher}. Here the regular 1-form $\omega_{\phi}$ of a model $\phi\in X_n^0$ is defined as in Section \ref{section genus one models and invariants}.

\begin{lem}
Let $C_{\phi}, C_{\phi'}$ be smooth curves of genus one over a field $K$ corresponding to the models $\phi, \phi' \in X_n^0$ respectively $(n\leq 5)$. Suppose $\phi'=g\phi$ for some $g\in \mathcal{G}_n$, then the isomorphism $\varphi : C_{\phi'} \rightarrow C_{\phi}$ determined by $g$ satisfies $\varphi^* \omega_{\phi}=(\det g) \omega_{\phi'}.$
\end{lem}
The author in \cite{Fisher} provides an explicit proof for $n\leq 5$ corresponding to the explicit linear algebraic groups $\mathcal{G}_n$ acting on $X_n$. We observe from this lemma that for any $g\in \mathcal{G}_n$ and any genus one curves $C_{\phi},C_{\phi'}$ defined by models $\phi, \phi'=g\phi$ in $X_n^0$, we have for a geometric modular form $\mathcal{F}$ of weight $k$
$$\mathcal{F}(C_{\phi},\omega_{\phi})=\mathcal{F}(C_{\phi'},\varphi^* \omega_{\phi})=\mathcal{F}(C_{\phi'},(\det g)\omega_{\phi'})=(\det g)^{-k} \mathcal{F}(C_{\phi'},\omega_{\phi'})$$
or
$$\mathcal{F}(C_{\phi'},\omega_{\phi'})=(\det g)^k \mathcal{F}(C_{\phi},\omega_{\phi}).$$
We have thus proved
\begin{prop}\label{proposition modular forms define invariants}
For $n\leq 5$, a geometric modular form $\mathcal{F}$ over a field $K$ defines an invariant of the same weight $I_{\mathcal{F}}$ over $K$ of $X_n^0$ such that $I_{\mathcal{F}}(\phi)=\mathcal{F}(C_{\phi},\omega_{\phi})$ for any $\phi\in X_n^0$.
\end{prop}

This fact over the complex numbers can also be seen directly from holomorphic modular forms. For a smooth curve $C$ given in the Weierstrass form $\phi$ such that $C(\mathbb{C})\cong \mathbb{C}/(\mathbb{Z}+\tau \mathbb{Z})$ for some $\tau\in \mathbb{H}$. Then $\phi$ is of the form $y^2-4x^3 + g_2x +g _3$, whose the two coefficients $g_2,g_3$ is written as
\begin{equation}\label{g2 g3 G4 G6}
g_2=60G_4, g_3=140G_6
\end{equation}
with $G_4,G_6$ defined in \eqref{G_2k}. The normalized invariants of the model $\phi$ are (see \cite[p. 367-368]{Villegas})
\begin{equation}\label{c4 c6 and g2 g3}
c_4(\phi)=2^6 3 g_2 \text{ , } c_6(\phi)=2^9 3^3 g_3.
\end{equation}
The following identities are then a consequence of \eqref{E_2k}, \eqref{g2 g3 G4 G6}, \eqref{c4 c6 and g2 g3} and the special values of zeta function $\zeta(4)=\pi^4/90, \zeta(6)=\pi^6/945$
\begin{equation}\label{E4 E6 c4 c6}
c_4(\phi)=(4\pi)^4 E_4(\tau), \text{  } c_6(\phi)=(4\pi)^6 E_6(\tau).
\end{equation}
It is then possible to compare the invariant $I_{\mathcal{F}}$ in case $\mathcal{F}=\mathcal{E}_4$ and $\mathcal{E}_6$ with the normalized ones $c_4,c_6$ as follows
\begin{lem}\label{lemma I_F is geometric invariant}
We have the following identities for any smooth model $\phi\in X_n^0$ ($n\leq 5$) over a field $K$
$$I_{\mathcal{E}_4}(\phi)=c_4(\phi), I_{\mathcal{E}_6}(\phi)=-c_6(\phi).$$
\end{lem}
\begin{proof}
Since both $I_{\mathcal{E}_{2k}}(\phi)$ and $c_{2k}(\phi)$ ($k=2,3$) depend only on the $K$-isomorphism class of the pair $(C_{\phi},\omega_{\phi})$ (see Definition \ref{defi:geommodularform} and \cite[Definition 2.1 and Proposition 5.23]{Fisher}), it is enough to consider the case $n=1$. We know that both $I_{\mathcal{E}_4}$ and $c_4$ are defined over $\mathbb{Z}$. Moreover, $c_4$ is a primitive polynomial. There exists thus a constant $\alpha\in \mathbb{Z}$ such that $I_{\mathcal{E}_4}(\phi)=\alpha c_4(\phi)$ for any $\phi\in X_1^0$. Consider a model $\phi$ with coefficients in $\mathbb{Z}\subset \mathbb{C}$, we obtain from \eqref{analytic and geometric} and \eqref{E4 E6 c4 c6} that $I_{\mathcal{E}_4}(\phi)=c_4(\phi)$ and hence $\alpha=1$. Here we are using the regular 1-form $\omega_{\phi}=dx/2y$ while the differential $dz$ in \eqref{analytic and geometric} is equal to $dx/y$. Similarly, we have $I_{\mathcal{E}_6}(\phi)=-c_6(\phi)$ for any $\phi\in X_1^0$.
\end{proof}
One can look closer at the connection between invariants and geometric modular forms. The author in \cite[Propositions 6.1, 6.2 and Remark 6.3]{Deligne} proved that $\mathcal{M}(K)$ is isomorphic to $K[X_1]^{G_1}$ over any field $K$. We want to get a similar phenomena when replacing $X_1$ by $X_n$ for any $n\leq 5$. Denote by $K[X_n^0]^{G_n}$ the ring of invariants of $X_n^0$ defined in the same way as in Definition \ref{definition ring of invariants}, i.e.,
$$K[X_n^0]^{G_n}:=\{I\in K[X_n^0] : I \circ g = I \text{ for all } g\in G_n(\overline{K})\}.$$
Since $K[X_n^0]=K[X_n][\Delta^{-1}]$, we conclude that $$K[X_n^0]^{G_n}=(K[X_n][\Delta^{-1}])^{G_n}=K[X_n]^{G_n}[\Delta^{-1}],$$
where the latter identity holds since $\Delta$ is an invariant in $K[X_n]^{G_n}$. The structure of $K[X_n^0]^{G_n}$ is deduced from Proposition \ref{proposition ring of invariants} with notations from there as follows.
\begin{prop}\label{proposition X_n^0 invariants}
The ring of invariants $K[X_n^0]^{G_n}$ of $X_n^0$ $(n\leq 5)$ over a field $K$ is
$$\left\{
  \begin{array}{ll}
    K[c_4,c_6][\Delta^{-1}], & \hbox{if char$(K)\neq 2,3$;} \\
    K[a_1,\Delta][\Delta^{-1}], & \hbox{if char$(K)=2$;} \\
    K[b_2,\Delta][\Delta^{-1}], & \hbox{if char$(K)=3$.}
  \end{array}
\right.
$$
\end{prop}

Proposition \ref{proposition modular forms define invariants} yields a ring homomorphism $I:\mathcal{M}^!(K)\rightarrow K[X_n^0]^{G_n}$ defined by $\mathcal{F}\mapsto I_{\mathcal{F}}$. There are, however, even more to this.

\begin{thm}\label{theorem invariants isomorphic modular forms}
Let $K$ be a field, the map $I: \mathcal{M}^!(K)\rightarrow K[X_n^0]^{G_n}$ $(n\leq 5)$ is an isomorphism.
\end{thm}
\begin{proof}
When char$(K)\neq 2$ or $3$, we know from Lemma \ref{lemma I_F is geometric invariant} that on $X_n^0$: $I_{\mathcal{E}_4}=c_4, I_{\mathcal{E}_6}=-c_6$ and $I_{\mathcal{D}}=\Delta$. Thus the ring homomorphism $I$ is bijective from Propositions \ref{proposition ring of geometric modular forms}, \ref{proposition X_n^0 invariants}, formula \eqref{M and M!} and the algebraic independence of $\mathcal{E}_4$ and $\mathcal{E}_6$, $c_4$ and $c_6$. Hence $I$ is an isomorphism.

In the case char$(K)=2$, $I$ sends $\mathcal{A}_2$ to $\alpha a_1$ for some constant $\alpha\in \mathbb{F}_2^*$, i.e., $\alpha=1$. This comes from the fact that $a_1$ is (up to constants) the only invariant of weight one. Moreover, $I$ sends $\mathcal{D}$ to $\Delta$ by Lemma \ref{lemma I_F is geometric invariant} and is thus an isomorphism from the algebraic independence of $\mathcal{A}_2$ and $\mathcal{D}$, $a_1$ and $\Delta$. The independence of $\mathcal{A}_2$ and $\mathcal{D}$ is deduced from the independence of $a_1$ and $\Delta$ since $I_{\mathcal{A}_2}=\alpha a_1$ and $I_{\mathcal{D}}=\Delta$ on $X_n^0$. The case char$(K)=3$ is treated similarly with the identities $I_{\mathcal{A}_3}=\beta b_2$ and $I_{\mathcal{D}}=\Delta$ on $X_n^0$. Here $\beta$ is some constant in $\mathbb{F}_3^*$.
\end{proof}

Back to our purpose, the key result in this section (which is Theorem \ref{theorem invariants} in Section \ref{section introduction}) is the following:
\begin{thm}\label{theorem invariants of C and E}
Let $\mathcal{F}$ be a geometric modular form of weight $k$, there exists a constant $\alpha_n \in K^*$ $(n\leq 4)$ such that
$$I_{\mathcal{F}}(\phi)=\alpha_n^k I_{\mathcal{F}}(\varphi_n(\phi))$$
for any smooth genus one model $(C_{\phi},\omega_{\phi})$ of degree $n$ over a field $K$ with the Jacobian $(E_{\phi},\omega_{\varphi_n(\phi)})$ constructed by the map $\varphi_n$. Moreover, $\alpha_2=\pm 1, \alpha_3=\pm 1/2,\alpha_4=\pm 2.$
\end{thm}
From Proposition \ref{proposition X_n^0 invariants} we know that the ring of invariants of $X_n^0$ is generated by elements of even weights except in the case of characteristic 2 in which $1=-1$. Therefore, we can in any case forget the about the sign of $\alpha_n$ and this gives a proof for Theorem \ref{theorem invariants}.
\begin{proof}
We consider the map $\varphi_n: X_n \rightarrow W$ defined in Section \ref{section discriminant and jacobians}, there exists $\alpha_n =\alpha_n(\phi)\in K^*$ depending on $\varphi_n$ and $\phi$ such that $\varphi_n^*(\omega_{\varphi_n(\phi)})=\alpha_n \omega_{\phi}$ for any $\phi\in X_n^0$. We have for any $\phi\in X_n^0$
$$I_{\mathcal{F}}(\varphi_n(\phi))=\mathcal{F}(E_{\phi},\omega_{\varphi_n(\phi)})=\mathcal{F}(C_{\phi}, \varphi_n^* \omega_{\varphi_n(\phi)})$$
$$=\mathcal{F}(C_{\phi},\alpha_n \omega_{\phi})=\alpha_n^{-k}\mathcal{F}(C_{\phi},\omega_{\phi})=\alpha_n^{-k}I_{\mathcal{F}}(\phi)$$
and hence
\begin{equation}\label{alpha_n}
I_{\mathcal{F}}(\phi)=\alpha_n^k I_{\mathcal{F}}(\varphi_n(\phi)).
\end{equation}
Since $\alpha_n$ does not depend on $k$, we can consider the case in which $k=12$ and $\mathcal{F}=\mathcal{D}$ to get the identities $\Delta_{\phi}=I_{\mathcal{D}}(\phi)$, $\Delta_{\varphi_n(\phi)}=I_{\mathcal{D}}(\varphi_n(\phi))$ from Lemma \ref{lemma I_F is geometric invariant} for the corresponding geometric modular form $\mathcal{D}=(\mathcal{E}_4^3 - \mathcal{E}_6^2)/1728$. This deduces
$$\Delta_{\phi}=\alpha_n^{12}\Delta_{\varphi_n(\phi)}.$$
As proved in Proposition \ref{proposition power of 2}, $\Delta_{\phi}=c\Delta_{\varphi_n(\phi)}$ for some constant $c$ when $n\leq 4$. More precisely,
$$\left\{
  \begin{array}{ll}
    c=\pm 1,  & \hbox{If $n=2$;} \\
    c=\pm 2^{-12},  & \hbox{If $n=3$;} \\
    c=\pm 2^{12},  & \hbox{If $n=4$.}
  \end{array}
\right.
$$
We now need to compute $\alpha_n$ from $\alpha_n^{12}=c$. Look again at \eqref{alpha_n} to the case $\mathcal{F}=\mathcal{E}_4$ and $\mathcal{E}_6$, we have $c_4(\phi)=\alpha_n^4 c_4(\varphi_n(\phi))$ and $c_6(\phi)=\alpha_n^6 c_6(\varphi_n(\phi))$ for any $\phi\in X_n^0$. Consider a model $\phi$ with integer coefficients, we conclude from Lemma \ref{primitivity of c4, c6 and D} that $\alpha_n^4, \alpha_n^6 \in \mathbb{Q}$ and thus $\alpha_n^2\in \mathbb{Q}$. This enables us to exclude the minus sign of the constant $c$ above and deduce the discussion at the end of Section \ref{section discriminant and jacobians}. We have in addition
$$\alpha_2^2=1, \alpha_3^2=1/4 \text{ and } \alpha_4^2=4$$
or
$$\alpha_2=\pm 1, \alpha_3=\pm 1/2 \text{ and } \alpha_4=\pm 2.$$

\end{proof}
Observe that the formulae of $c_4,c_6$ obtained from Theorem \ref{theorem invariants} for $k=4,6$ respectively in cases $n=2,3,4$ are the same with the normalized ones given by Fisher as in \eqref{n=2 geometric}, \eqref{n=3 geometric} and \eqref{n=4 geometric}.

\section*{Acknowledgement}

I would like to thank my advisor Dennis Eriksson for introducing me to the topic, providing me with many important ideas and materials as well as crucial corrections and comments. I am also grateful to my co-advisor Martin Raum for his useful discussions and feedback.

\end{document}